\documentclass{amsart}
\usepackage{amssymb}
\newtheorem{thm}{Theorem}[section]

\newtheorem{lem}[thm]{Lemma}
\newtheorem{prop}[thm]{Proposition}
\theoremstyle{definition}

\newtheorem{defn}[thm]{Definition}

\newtheorem*{defn*}{Definition}
\newtheorem*{rems*}{Remarks}
\newtheorem*{rem*}{Remark}

\numberwithin{equation}{section}

\begin{document}

\title [Generic singularities of
symplectic immersions] {Generic singularities of symplectic \\ and
quasi-symplectic immersions }
\author{W. Domitrz}
\address{Warsaw University of Technology\\
Faculty of Mathematics and Information Science\\
Plac Politechniki 1\\
00-661 Warsaw\\
Poland }

\email{domitrz@mini.pw.edu.pl}

\author{S. Janeczko}
\address{Institute of Mathematics\\
Polish Academy of Sciences\\
Sniadeckich 8\\
P.O. Box 137\\
00-950 Warsaw\\
Poland\\
Warsaw University of Technology\\
Faculty of Mathematics and Information Science\\
Plac Politechniki 1\\
00-661 Warsaw\\
Poland  }

\email{janeczko@alpha.mini.pw.edu.pl}

\author{M. Zhitomirskii }
\address{Department of Mathematics\\
Technion\\
32000 Haifa\\
Israel }

\email{mzhi@techunix.technion.ac.il}
\thanks{The research of W. Domitrz and S. Janeczko was supported by Polish MNiSW grant \newline
\hskip 2cm no. N N201
397237. The research of M. Zhitomirskii was supported by the Israel Science Foundation grant 1383/07.}

 \maketitle

\begin{abstract} For any $k<2n$ we construct a complete system of invariants
in the problem of classifying singularities of immersed
$k$-dimensional submanifolds of a symplectic $2n$-manifold at a
generic double point.
\end{abstract}

\section{Introduction}.
\label{sec-intr}

\subsection{Symplectic and quasi-symplectic immersions}
\label{subsec-immersions}
A smooth $2r$-dimensional submanifold $S$ of a $2n$-dimensional  symplectic manifold $(M^{2n}, \omega )$ is called symplectic
if the restriction $\omega \vert _{TS}$ has the maximal
possible rank $2r$. If $dim S = 2r+1$ then the maximal possible rank of this  restriction
is also $2r$ and in this case $S$ is called quasi-symplectic.

\medskip

The Darboux-Givental theorem (see \cite{AG}) states that in the
problem of local classification of pairs
consisting of a symplectic form on $M^{2n}$  and a
smooth submanifold of $M^{2n}$, the pullback of the symplectic
form to the submanifold is a complete invariant. This theorem
implies that any two germs of smooth symplectic or quasi-symplectic submanifolds of the same dimension of a symplectic manifold
can be brought one to the other by a local diffeomorphism preserving the
symplectic form.

\medskip

The present work is devoted to the classification of first occurring singularities of
immersed symplectic or quasi-symplectic submanifolds of a symplectic manifold, i.e. classification of the
tuples
\begin{equation}
\label{tu} \left(\mathbb R^{2n}, \ \omega , \ S_1^k \cup S_2^k \right)_0
\end{equation}
where $\omega $ is a symplectic form on $\mathbb R^{2n}$ and  $S_1^k, S_2^k$ are $k$-dimensional symplectic or quasi-symplectic  submanifolds of $\left(\mathbb R^{2n}, \omega \right)$ whose intersection contains $0\in \mathbb R^{2n}$.
The notation $(\ \ )_0$ means that objects in the parenthesis are germs at $0\in \mathbb R^{2n}$.
A tuple (\ref{tu}) is equivalent to a tuple of the same form with $\widetilde \omega , \widetilde S_1^k, \widetilde S_2^k$ if there exists a local diffeomorphism of $\mathbb R^{2n}$ which brings $\widetilde \omega $ to $\omega $ and $S_1^k \cup S_2^k$ to $\widetilde S_1^k \cup \widetilde S_2^k$.
We work in a fixed category which is smooth or real-analytic.
We restrict ourselves to generic germs of (\ref{tu}) which means that our results
 concern a certain open and dense set in the space of such germs.

\subsection{The cases of dimension $1$ and codimension $1$} Note that any hypersurface and any 1-dimensional submanifold of a symplectic manifold are quasi-symplectic.
Within generic germs the cases $k=1$ and $k=2n-1$ are much simpler than the case $2\le k\le n-2$ and in these two
cases the classification of generic tuples (\ref{tu}) is contained in the work \cite{Ar1} by V. Arnol'd $(k=1)$ and
in the work \cite{Me} by R. B. Melrose $(k=2n-1)$. Theorems \ref{thm-dim-1} and \ref{thm-codim-1} below are the simplest
particular results of these works.

\begin{thm}
\label{thm-dim-1} Let $k=1$. All germs (\ref{tu}) with non-tangent strata $S_1^1, S_2^1$ and such that
the restriction of $\omega $ to $T_0S_1^1+T_0S_2^1$ has maximal rank $2$ are
equivalent.
\end{thm}

This theorem is the simplest case of the
symplectic classification of singular curves diffeomorphic to
$A_{\mu }=\{x\in\mathbb R^{2n}:x_1^{\mu+1}-x_2^2=x_{\ge 3}=0\}$
obtained by V. Arnol'd in \cite{Ar1}, namely the case $\mu=1$. All germs (\ref{tu}) with $k=1$, non-tangent
strata $S_1^1, S_2^1$ and such that $\omega $ annihilates the space $T_0S_1^1+T_0S_2^1$ are also
equivalent. In the case of the tangent strata, with a finite order of tangency, the classification
is more involved, but remains discrete.
These results from the work \cite{Ar1} are explained in \cite{DJZ2} using the method of algebraic restrictions,
developed in \cite{Zh1} for classification of singular varieties in a contact space and in \cite{DJZ2} for
classification of singular varieties in a symplectic space. The work \cite{DJZ2} contains symplectic
classification of singular curves with any fixed $A$ or $D$ or $E$ singularity.

\begin{thm}
\label{thm-codim-1}
Let $k=2n-1$. All germs (\ref{tu}) with transversal hypersurfaces $S_1^{2n-1}, S_2^{2n-1}$  such that
the restriction of $\omega $ to $T_0S_1^{2n-1}\cap T_0S_2^{2n-1}$ has maximal rank \ $2n-2$ are
equivalent.
\end{thm}

This theorem  was proved by R. B.  Melrose in \cite{Me}. See the proof of Proposition 2.1 in this work, where
Theorem \ref{thm-codim-1} is formulated in a different, but equivalent form.
The main part of the work \cite{Me} is devoted to a much more difficult case that $S_1^{2n-1}, S_2^{2n-1}$
are transversal, but the restriction of $\omega $ to the manifold $S_1^{2n-1}\cap S_2^{2n-1}$
has the first occurring singularity within closed 2-forms on an even-dimensional manifold, so called $\Sigma_{20}$ singularity studied by J. Martinet (see \cite{Ma} or Appendix G of \cite{Zh2}). In particular, the restriction of $\omega $ to
$T_0S_1^{2n-1}\cap T_0S_2^{2n-1}$ has rank $2n-4$. In this case, expressed in \cite{Me} is equivalent
terms of the Poisson bracket of
the functions $f_1, f_2$ defining $S_1^{2n-1}, S_2^{2n-1}$, these hypersurfaces are called glancing. Melrose proved that in the $C^\infty $ category all tuples (\ref{tu})
 with glancing hypersurfaces are equivalent. In analytic category it is not so as it was showed in
\cite{Os}.

\subsection{Moduli in the case $2\le k\le 2n-2$}
 \label{subsec-moduli} In this case already the classification of generic tuples (\ref{tu}) is a much harder problem. The only result we know concerns the case $k=2$ in our work \cite{DJZ2}, section 7.4.
In the present work we classify generic tuples (\ref{tu}) for any $k$ and $n$. Our theorems  on a complete system of invariants in section \ref{sec-main-results} imply the following statement.

\begin{thm}
\label{thm-moduli}
In the problem of classification of generic tuples (\ref{tu}) with $2\le k\le 2n-2$ there are $[k/2]$ moduli if
$2\le k\le n$, there is one modulus if $k=2n-2$ or $k = 2n-3$, and in the remaining case $n<k\le 2n-4$ {\rm (}which is possible for $2n\ge 10$ only{\rm )} there are functional moduli which belong to the space of tuples of  $(s-1)$ functions of $d$ variables, where $s = [(2n-k)/2]$ and $d = 2(k-n)$.
\end{thm}

The precise meaning of the last statement, about the functional moduli, is explained in Theorem \ref{thm-D}, section \ref{sec-main-results}.

\subsection{Tools}
\label{subsec-tools}
 Our starting point is the following proposition.

\begin{prop}[\cite{DJZ2}]
\label{prop-basic-AA}
Let $N = S_1\cup S_2\cup \cdots \cup S_r$ where $S_i$ are germs at $0$ of smooth submanifolds of $\mathbb R^{2n}$ such that
\begin{equation}
\label{regular-union}
dim (T_0S_1 + \cdots + T_0S_r) = dim S_1 + \cdots + dim S_r.
\end{equation}
Let $\omega $ and $\widetilde \omega $ be
symplectic forms on $\mathbb R^{2n}$ with the same restriction to the tangent bundles of $S_i^k$
  and the same restriction to the space $T_0S_1 + \cdots + T_0S_r$. There exists a local diffeomorphism of $\mathbb R^{2n}$
which sends $\widetilde \omega $ to $\omega $ preserving pointwise $N$.
\end{prop}

Strictly speaking, this proposition is not formulated in \cite{DJZ2}, but it is a logical corollary
of two results from this work. The first one, Theorem A in section 2.7 of \cite{DJZ2}, states
that given a germ
of any quasi-homogeneous variety $N\subset \mathbb R^{2n}$ (in particular $N$ in Proposition \ref{prop-basic-AA}) any two symplectic forms $\omega $ and $\widetilde \omega $ on $\mathbb R^{2n}$ with the same {\it algebraic restriction} to $N$ can be brought one to the other by a
local diffeomorphism of $\mathbb R^n$ which preserves $N$ pointwise. We refer to \cite{DJZ2} to the definition of algebraic restrictions and the method of algebraic restrictions for local classification of singular varieties in a symplectic manifold, and we refer to \cite{DJZ1} for the definition of a quasi-homogeneous variety and its role
in local analysis.  Proposition \ref{prop-basic-AA} is a logical corollary of the formulated theorem and
another result from \cite{DJZ2}, Theorem 7.1  stating that under assumptions of Proposition
 \ref{prop-basic-AA} the symplectic forms $\omega $ and $\widetilde \omega $ have the same algebraic restriction to $N$.

\medskip

In the case $k>n$ we also use the following result by Alan S. McRae.

\begin{prop}[\cite{MR}]
\label{prop-basic-BB}
Let $S_1$ and $S_2$ be germs at $0$ of submanifolds of \ $\mathbb R^{2n}$ such that $T_0S_1+T_0S_2 = T_0\mathbb R^{2n}$. Let $\omega $ and $\widetilde \omega $ be symplectic forms on $\mathbb R^{2n}$
coinciding at any point $z\in S_1\cap S_2$ and having the same restrictions to $TS_1$ and $TS_2$. There exits a local diffeomorphism of $\mathbb R^{2n}$ preserving pointwise $S_1$ and $S_2$ and bringing $\widetilde \omega $ to $\omega $.
\end{prop}

In fact, McRae proved a bit stronger result: Proposition \ref{prop-basic-BB} holds not only locally, but also
in a neighbourhood of the union $S_1\cup S_2$ provided $S_1, S_2$ are closed and $\omega $ can be deformed into
$\widetilde \omega $ inside the class of symplectic structures having the properties in Proposition \ref{prop-basic-BB}. The latter certainly holds if $S_1$ and $S_2$ are germs at $0$, for the deformation
$\omega _t = \omega + t(\widetilde \omega - \omega )$, $t\in [0,1]$.

\medskip

We also need the following proposition which is a slight generalization of the Darboux-Givental' theorem.

 \begin{prop}
 \label{prop-Darboux}
 Let $\mu $ and $\widetilde \mu $ be the germs at $0$ of closed 2-forms on $\mathbb R^k$
 of maximal rank $2[k/2]$ such that $\mu (z) = \widetilde \mu (z)$ for any point $z$ of a submanifold
  $Q\subset \mathbb R^k$. If $k$ is odd we assume that the lines $ker \mu (0)$ and $ker \widetilde \mu (0)$
  do not belong to $T_0Q$. Then $\widetilde \mu $ can be brought to $\mu $ by a local
  diffeomorphism of $\mathbb R^k$ which preserves $Q$ pointwise and has identity linear approximation at any point of $Q$.
 \end{prop}

\begin{proof}
In the even-dimensional case it is exactly the Darboux-Givental' theorem up to
the assumption that $\mu $ and $\widetilde \mu $ agree at points of $Q$ and the requirement that
the reducing diffeomorphism has identity linear approximation at points of $Q$. The proof is exactly
the same as the proof of
the Darboux-Givental' theorem in \cite{AG}. The odd-dimensional case reduces to the even-dimensional
case as follows. Take a hypersurface $H$ which contains $Q$ and which is transversal to
the kernels of $\mu$ and $\widetilde \mu$.
The restrictions of $\mu$ and $\tilde \mu$ to $TH$
are symplectic. Take a local diffeomorphism
$\widehat \Phi $ of $H$ which preserves $Q$ pointwise, brings $\tilde \mu \vert _{TH}$ to $\mu \vert _{TH}$, and has
identity linear approximation at any point of $Q$. Take vector
fields $X$ and $\widetilde X$ which generate the kernels of $\mu$, $\widetilde \mu$ respectively and agree at any point of $Q$. Let $\Psi^t$ and $\widetilde \Psi^t$ be the flows of $X$ and $\widetilde X$.
  The required local
diffeomorphism
 $\Phi $ of $\mathbb R^k$ can be constructed as follows: for any point $p\in \mathbb R^k$, close to $0$, we
 take $t = t(p)$ such that $\tilde \Psi^{t(p)} \in H$ and  we set
 $\Phi(p)=(\Psi^{-t(p)} \circ \widehat \Phi \circ \tilde \Psi^{t(p)})(p)$.
\end{proof}

\medskip

Finally, we need a simple part of classification of couples of symplectic forms on the same vector space.
This classification problem was solved in \cite{GZ} by I. Gelfand and I. Zakharevich. We need the following statement formulated in terms of skew-symmetric matrices.

\begin{prop}[\cite{GZ}, section 1]
\label{prop-GZ}
Let $A$ and $B$ be  non-singular skew-symmetric $2s\times 2s$
matrices. The tuple of eigenvalues of the matrix $A^{-1}B$ is
an invariant of the couple $(A,B)$ with respect to the group of
transformations $(A,B)\ \to \ (R^tAR, \ R^tBR)$, ${\rm det} R\ne
0$. The multiplicity of each of the eigenvalues of the
matrix $A^{-1}B$ is greater than $1$ and consequently this matrix has not
more than $s$ distinct eigenvalues. It has exactly $s$ eigenvalues for a generic
couple $A$ and $B$. In this case the tuple of eigenvalues of $A^{-1}B$ is a complete
invariant of $(A,B)$.
\end{prop}

\subsection{Structure of the paper}
In section \ref{sec-linearization} we present linearization theorems which can be easily proved using
Propositions \ref{prop-basic-AA} - \ref{prop-Darboux}. We believe that one of the main contribution of this
work, maybe the main one, is construction of invariants of tuples (\ref{tu}) which we call characteristic numbers.
The characteristic numbers are constructed in section \ref{sec-char-numbers}. In the case $k>n$ a generic tuple (\ref{tu}) defines a manifold $Q=S_1^k\cap S_2^k$ endowed with a symplectic form $\omega \vert _{TQ}$ and
characteristic numbers can be extended to characteristic Hamiltonians on the symplectic manifold $Q$.
The tuple of characteristic Hamiltonians, also constructed in section \ref{sec-char-numbers}, is an invariant of
(\ref{tu}) with $k>n$ up to a symplectomorphism of $Q$. Our final theorems on {\it complete} system of invariants
are contained in section \ref{sec-main-results}, along with normal forms following from these theorems.

\section{Linearization theorems}
\label{sec-linearization}

\subsection{Regular intersection of $S_1^k$ and $S_2^k$} Our final theorems in section \ref{sec-main-results}  hold under certain  genericity condition, which we call the regularity of a tuple (\ref{tu}). It includes
 the regularity of the intersection of the strata $S_1^k, S_2^k$.

\begin{defn}
\label{def-regular-intersection}
The strata $S_1^k, S_2^k$ in (\ref{tu}) have regular intersection if

\noindent $T_0S_1^k\cap T_0S_2^k=\{0\}$ for  $k\le n$ and $T_0S_1^k+T_0S_2^k = T_0\mathbb R^{2n}$ for  $k>n.$
\end{defn}

\subsection{Linearization} The regularity of the intersection of the strata is a property of the linearization of (\ref{tu}) which is a tuple
\begin{equation}
\label{tu-spaces}
\left(W^{2n}, \ \sigma , \ U_1^k\cup U_2^k\right)
\end{equation}
consisting  of a $2n$-dimensional vector space $W^{2n}$, a symplectic (i.e. non-degenerate) 2-form $\sigma $ on $W^{2n}$, and the union of the $k$-dimensional subspaces $U_1^k, U_2^k$.

\begin{defn}
\label{def-linearization}
The linearization of a tuple $(\mathbb R^{2n}, \omega , S_1^k\cup S_2^k)$ at a point $z\in S_1^k\cap S_2^k$ is the tuple (\ref{tu-spaces}) with
$W^{2n} = T_z\mathbb R^{2n}, \ \sigma = \omega \vert _{W^{2n}}$ and $U_i^k = T_zS_i^k$, \ $i = 1,2$.
\end{defn}

\subsection{Tuples (\ref{tu}) with the same linearization} The following two theorems can be easily proved
using Propositions \ref{prop-basic-AA} - \ref{prop-Darboux}.

\begin{thm}
\label{thm-basic-A}
Two tuples (\ref{tu}) with the same regularly intersecting symplectic or quasi-symplectic strata $S_1^k, S_2^k$ of dimension $k\le n$ and the same linearization at $0\in \mathbb R^{2n}$ are equivalent.
\end{thm}

\begin{thm}
\label{thm-basic-B}
Two tuples (\ref{tu}) with the same regularly intersecting symplectic or quasi-symplectic strata $S_1^k, S_2^k$ of dimension $k>n$ and the same linearization at any point $z\in S_1^k\cap S_2^k$ close to $0\in \mathbb R^{2n}$ are equivalent provided that the restrictions of $\omega $ and $\widetilde \omega $ to $T_0S_1^k\cap T_0S_2^k$ have the maximal rank \ $2(n-k)$.
\end{thm}

\begin{proof} Theorem \ref{thm-basic-A} can be reduced to Proposition \ref{prop-basic-AA} with $r=2$ as follows.
Since the linearizations of the tuples are the same, we have $\omega (0) = \widetilde \omega (0)$. By Proposition \ref{prop-Darboux} with $Q = \{0\}$ there exist local diffeomorphisms $\phi _i$ of $S_i^k$, $i=1,2$ with identity linear approximations at $0$ which bring the restriction of $\widetilde \omega $ to $TS_i^k$ to the restriction of
$\omega $ to $TS_i^k$, $i=1,2$.
 We can construct a local diffeomorphism $\Phi $ of $\mathbb R^{2n}$, also with identity linear approximation at $0$, which preserves $S_i^k$ and whose restriction to $S_i^k$ coincide with $\phi _i$, $i=1,2$. The diffeomorphism $\Phi $ brings $\widetilde \omega $ to a symplectic form $\widehat \omega $ such that $\omega $ and $\widehat \omega $ have the same restriction to $TS_i^k$ and the same restriction to the
 space $T_0\mathbb R^{2n}$. Proving the equivalence of tuples (\ref{tu}) we may replace $\widetilde \omega $ by $\widehat \omega $. Now the equivalence follows from Proposition \ref{prop-basic-AA}.

\medskip

Theorem \ref{thm-basic-B} can be reduced to Proposition \ref{prop-basic-BB} in exactly the same way using
Proposition \ref{prop-Darboux} with $Q = S_1^k\cap S_2^k$. If $k$ is odd, we have a right to use Proposition \ref{prop-Darboux}
because the one-dimensional kernels of forms $\omega (0)$ and $\widetilde \omega (0)$ are not tangent  to $Q$. It follows from the assumption in Theorem \ref{thm-basic-B} that
the restrictions of $\omega $ and $\widetilde \omega $ to $T_0Q$ have the maximal rank.
\end{proof}

\subsection{Isomorphic tuples (\ref{tu-spaces}). Linearization theorem.}

Theorem \ref{thm-basic-A} implies the following linearization theorem involving the natural definition of
isomorphic tuples (\ref{tu-spaces}).

\begin{defn}
\label{def-isomorphic}
A tuple (\ref{tu-spaces})
is isomorphic to a tuple of the
same form with $\widetilde W^{2n}, \ \widetilde \sigma, \ \widetilde U_1^k\cup \widetilde U_2^k$ \ if there exists
an isomorphism from $W^{2n}$ to $\widetilde W^{2n}$ sending $\widetilde \sigma $ to $\sigma $ and sending $U_1^k\cup U_2^k$ \ to
\ $\widetilde U_1^k\cup \widetilde U_2^k$.
\end{defn}

\begin{thm}
\label{thm-linearization}
If two tuples (\ref{tu}) with regularly intersecting symplectic or quasi-symplectic strata are
equivalent then their linearizations at $0\in \mathbb R^{2n}$ are isomorphic. In the case $k\le n$
the tuples are equivalent if and only if their linearizations at $0\in \mathbb R^{2n}$ are isomorphic.
\end{thm}

\begin{proof} The first statement follows from the observation that if two tuples (\ref{tu}) are equivalent via
a local diffeomorphism $\Phi $ then their linearizations at $0$ are isomorphic via the isomorphism $d\Phi\vert _0$.
The second statement is a direct corollary of Theorem \ref{thm-basic-A} and the fact that for $k\le n$ any pair of germs at $0$ of smooth $k$-dimensional submanifolds regularly intersecting is diffeomorphic to its linearization by a diffeomorphism with identity linear approximation at $0$.
\end{proof}

Using Theorem \ref{thm-basic-B} we could formulate a linearization theorem for the case $k>n$, with necessary and sufficient rather than only necessary condition for the equivalence of tuples (\ref{tu}), but the formulation of such a theorem is rather involved, and we do not need it for the proof of our final theorem for the case $k>n$, we use just Theorem \ref{thm-basic-B}.

\section{Characteristic numbers and characteristic Hamiltonians}
\label{sec-char-numbers}

\subsection{Regular tuples (\ref{tu}) and (\ref{tu-spaces})} By Theorem \ref{thm-linearization} the problem of classifying tuples (\ref{tu}) reduces to the problem of
classifying tuples (\ref{tu-spaces}) with respect to isomorphisms if $k\le n$ and
contains this problem if $k>n$.
We solve this problem for generic tuples (\ref{tu-spaces}), namely for tuples (\ref{tu-spaces})
satisfying the following conditions.

\begin{defn}
\label{def-regular}
A tuple (\ref{tu}) will be called {\it regular} if its linearization at $0\in \mathbb R^{2n}$, a tuple of form (\ref{tu-spaces}), is regular.
A tuple (\ref{tu-spaces}) is regular if its ingredients satisfy the following requirements.

\medskip

\noindent 1. The subspaces $U_1^k$ and $U_2^k$ are symplectic or quasi-symplectic, with regular intersection:
 $U_1^k\cap U_2^k=\{0\} \ \text{for}\  k\le n;  \ \ U_1^k+U_2^k = W^{2n} \ \text{for} \ k>n.$

\medskip

\noindent 2. If $k\le n$ the restriction of $\sigma $ to $U_1^k+U_2^k$ has maximal rank $2k$. If $k>n$                 the restriction of $\sigma $ to $U_1^k \cap U_2^k$ has maximal rank $2(k-n)$.
\medskip

\noindent 3. The skew-orthogonal complement to $U_1^k$ in $\left(W^{2n}, \sigma \right)$ is transversal to   $U_2^k$.
\medskip

\noindent 4. This condition is required only for odd $k$. In this case the previous conditions imply that $\ell _i =   \label{tu-alg-reduced}
ker \ \sigma \vert _{U_i^k}$, \ $i=1,2$  are different 1-dimensional subspaces of $W^{2n}$. We require that the 2-form $\sigma $ does not annihilate the plane $\ell _1+\ell _2$.

\end{defn}

\subsection{Reduction of dimensions} Our first step in classifying regular tuples (\ref{tu-spaces}) with $2\le k\le 2n-2$ is reduction of dimensions $2n, k$ to
$4s, 2s$. Namely we associate to a regular tuple (\ref{tu-spaces}) a tuple
\begin{equation}
\label{tu-spaces-reduced}
\left(\widehat W^{4s}, \ \widehat \sigma , \ \widehat U_1^{2s}\cup \widehat U_2^{2s}\right)
\end{equation}
\begin{equation}
\label{s}
s = s(k,n) = {\rm
min}\ \big([k/2], \ [(2n-k)/2]\big)
\end{equation}
constructed as follows, where
$\ell _i = ker \ \sigma \vert _{U_i^k}$ and the sign $^\perp $ denotes the skew-orthogonal
complement in the symplectic space $(W^{2n}, \sigma)$:

\begin{equation*} \label{W-even}
 k \ \text{even}: \ \ \ \ \ \widehat W^{4s}  =
\begin{cases} U_1^k + U_2^k\ \ \ \hskip 1cm \text{if} \ k\le n \\
\left(U_1^k \cap U_2^k\right)^{\perp }  \ \ \ \ \ \text{if}\
k>n;\end{cases}\hskip 2.4cm \end{equation*}
\begin{equation*}
\label{W-odd} k \ \text{odd}: \ \ \ \ \ \ \widehat W^{4s}  =
\begin{cases} \left(U_1^k + U_2^k \right)\cap (\ell _1 + \ell _2)^{\perp }\ \ \ \hskip 0.7cm \text{if} \ k\le n \\ \left(U_1^k
\cap U_2^k\right)^{\perp }\cap (\ell _1 + \ell _2)^{\perp  } \ \ \
\ \ \text{if}\ k>n
\end{cases} \end{equation*}
and for
any parity of $k$ we set
 $$\widehat \sigma  = \sigma \vert
_{\widehat W^{4s}},  \ \ \widehat U_i^{2s} = U_i^k\cap \widehat W^{4s},  \ i = 1,2.$$

\begin{prop}
\label{prop-reduced-tuple}
For any regular tuple (\ref{tu-spaces}) the dimension of $\widehat W^{4s}$ is $4s$, the dimension of $\widehat U_i^{2s}$ is  $2s$ and the form
$\widehat \sigma $ on $\widehat W^{4s}$ is symplectic so that the tuple (\ref{tu-spaces-reduced}) has the same form as the tuple (\ref{tu-spaces}). The tuple (\ref{tu-spaces-reduced}) is also regular, i.e. it satisfies all the requirements in Definition \ref{def-regular}. Two regular tuples (\ref{tu-spaces}) are isomorphic if and only if so are the corresponding reduced tuples (\ref{tu-spaces-reduced}).
\end{prop}

\begin{defn}
\label{def-reduced}
The constructed tuple (\ref{tu-spaces-reduced}) will be called
reduced tuple, associated with a regular tuple (\ref{tu-spaces}).
\end{defn}

Proposition \ref{prop-reduced-tuple}, reducing classification of regular tuples (\ref{tu-spaces}) to the classification of
regular tuples (\ref{tu-spaces-reduced}),  is a simple statement and we leave its proof to a reader. The proof requires not more than the linear Darboux theorem stating that the rank of a 2-form on a  vector space is its complete invariant with respect to isomorphisms.

\subsection{Two linear operators defined by reduced tuples (\ref{tu-spaces-reduced})} Our next step is the construction of two linear operators
associated with such tuples. The regularity of (\ref{tu-spaces-reduced}) implies that
we have the direct sums $\widehat W^{4s} = \widehat U_1^{2s}\oplus \widehat U_2^{2s} = \widehat U_1^{2s} \oplus (\widehat U_1^{2s})^{\perp} = \widehat U_2^{2s} \oplus (\widehat U_2^{2s})^{\perp}$ where
as above the sign $\perp $ means the skew-orthogonal complement in the space $(\widehat W^{4s}, \sigma )$.
Consider  the projections associated with the last two direct sums:

\smallskip

$$\widehat W^{4s} = \widehat U_1^{2s} \oplus (\widehat U_1^{2s})^{\perp} , \ \ \pi _1 :  \widehat W^{4s}\to
\widehat U_1^{2s},$$ $$\widehat W^{4s} = \widehat U_2^{2s} \oplus (\widehat U_2^{2s})^{\perp} , \ \ \pi _2 : \widehat W^{4s}\to \widehat U_2^{2s}.$$

Define linear operators $L_1: \widehat U_1^{2s}\to \widehat U_1^{2s}$ and $L_2: \widehat U_2^{2s}\to \widehat U_2^{2s}$ by
the diagram

\medskip

$\hskip 5.2cm _{L_1}$

$ \hskip 4cm  \widehat U_1^{2s}  \ \ \ \ \longrightarrow \ \ \ \ \widehat U_1^{2s}$ \hskip
2.2cm $L_1 = \pi _1\circ  (\pi _2\vert _{\widehat U_1^{2s}})$

 $\hskip 4cm \searrow _{\pi _2 } \hskip .5cm  _{\pi _1}\nearrow \ \  \ \searrow
 \ _{\pi _2}$ \ \ \ :

\medskip

$\hskip 4.8cm \widehat U_2^{2s} \  \ \ \ \longrightarrow \ \ \ \ \widehat U_2^{2s} $ \hskip
1.4cm  $L_2 = \pi _2 \circ (\pi _1 \vert _{\widehat U_2^{2s}})$

 $\hskip 6 cm ^{L_2}$

\begin{lem}
\label{lem-op} For any regular tuple (\ref{tu-spaces-reduced}) the linear
operators $L_1$ and $L_2$ are conjugate and consequently have the
same eigenvalues.
\end{lem}

\begin{proof} Note that the given diagram implies that the diagram

\hskip 4.9cm $_{L_1}$

$\hskip 4cm \widehat U_1^{2s} \ \ \longrightarrow \ \ \widehat U_1^{2s}$

$\hskip 3.8cm \pi _2\downarrow \hskip 1.2cm \downarrow \pi _2$

\medskip

$\hskip 4cm \widehat U_2^{2s} \ \ \longrightarrow \ \ \ \widehat U_2^{2s}$

\hskip 4.9cm $^{L_2}$

\noindent is commutative. Since (\ref{tu-spaces-reduced}) is a regular tuple, the three spaces $\widehat U_1^{2s}, \widehat U_2^{2s},
(\widehat U_1^{2s})^{\perp }$ are transversal one to the other. It follows that
$\pi _2$ restricted to $\widehat U_1^{2s}$ is a bijection between $\widehat U_1^{2s}$ and
$\widehat U_2^{2s}$. \end{proof}

\subsection{Characteristic numbers}

\begin{defn}
\label{def-charact-numbers}
Let (\ref{tu}) be a regular tuple, (\ref{tu-spaces}) is its linearization at $0\in \mathbb R^{2n}$
and (\ref{tu-spaces-reduced}) the reduced linearization.
The eigenvalues (real and complex) of the constructed linear operators $L_1$ or $L_2$ will be called {\it characteristic numbers}
of the tuples (\ref{tu}), (\ref{tu-spaces}), (\ref{tu-spaces-reduced}).
\end{defn}

The following statement is a direct corollary of Theorem \ref{thm-linearization} and Proposition \ref{prop-reduced-tuple}.

\begin{thm}
\label{thm-char-numbers}
The characteristic numbers of a regular tuple (\ref{tu}) are its invariants: if two regular tuples of form (\ref{tu}) are equivalent then their characteristic numbers are the same.
\end{thm}

The linear operators $L_1$ and $L_2$ are defined on vector spaces of dimension $2s$ and from the first glance it seems that a generic regular tuple (\ref{tu-spaces-reduced}) and consequently a generic
regular tuple (\ref{tu}) has $2s$ distinct characteristic numbers. It is not so.
The matrix of $L_1$, resp. $L_2$ in some and then any basis of the vector space
$\widehat U_1^{2s}$, resp. $\widehat U_2^{2s}$ is the product of two skew-symmetric $2s\times 2s$ matrices, and the eigenvalues of such
matrices are not generic in the space of tuples of $2s$ complex numbers.
To explain this claim, take any basis $B_1=(u_{1,1},...u_{1,2s})$ of $\widehat U_1^{2s}$ and any basis $B_2=(u_{2,1},...u_{2,2s})$ of $\widehat U_2^{2s}$. The 2-form $\sigma $ on $\widehat W^{4s}$ is defined by a $4s\times 4s$ skew-symmetric matrix of the form
\begin{equation}
\label{matrices}
\sigma : \ \ \begin{pmatrix}A_1&C\cr C^t&A_2\end{pmatrix}, \ \ A_1,A_2,C\in Mat(2s\times 2s), \ \ A_1^t = -A_1, \ A_2^t = -A_2.
\end{equation}
The matrices $A_1$ and $A_2$ are non-singular. Since the tuple (\ref{tu-spaces-reduced}) is regular, the skew- orthogonal complement to $\widehat U_1^{2s}$
is transversal to $\widehat U_2^{2s}$ and it follows
that the matrix $C$ is also non-singular. The latter allows to change the basis $B_1$ by the transition
matrix $C^{-1}$ to a new  basis $\widetilde B_1$ of $\widehat U_1^{2s}$ so that in the basis $(\widetilde B_1, B_2)$ the 2-form $\sigma $ is defined by matrix (\ref{matrices}) with $C = I$ (certainly the matrices $A_1$ and $A_2$ will change).
After this reduction of $C$ to $I$, it is not hard to compute the matrix of the linear operator $L_1$ in the basis $\widetilde B_1$, it is the skew-symmetric matrix $A_1^{-1}A_2$.

\medskip

It is easy to see that when changing the {\it both} basis $B_1$ and $B_2$, the matrix $C$ in
(\ref{matrices}) remains identity if and only if
the transformations of $B_1$ and $B_2$ are defined by matrices $R$ and $(R^t)^{-1}$, where $R$ is any non-singular $2s\times 2s$ matrix. Such transformations of $B_1$ and $B_2$ bring the matrices $A_1$ and $A_2$ in
(\ref{matrices}) to the matrices $A_1\to R^tA_1R, \ A_2\to R^tA_2R$.

\medskip

The outcome of this linear algebra computation (expressed without details which we leave to a reader) is as follows.

\begin{prop}
\label{prop-algebra-1}
One can associate to any regular tuple (\ref{tu-spaces-reduced}) two non-singular skew-symmetric $2s\times 2s$ matrices $A_1, A_2$ so that the characteristic numbers of (\ref{tu-spaces-reduced}) are the eigenvalues of the
matrix $A_1^{-1}A_2$ and two tuples (\ref{tu-spaces-reduced}) are isomorphic if and only if
the corresponding couples of skew-symmetric matrices can be brought one to the other by a transformation
$(A_1,A_2)\to (R^tA_1R, \ R^tA_2R)$, $det R\ne 0$. Any couple $(A_1,A_2)$ with two non-singular skew-symmetric $2s\times 2s$ is realizable, i.e. it is associated to some regular tuple (\ref{tu-spaces-reduced}).
\end{prop}

Consequently the classification of regular
tuples (\ref{tu-spaces-reduced}) is exactly the same problem as the classification
of couples of symplectic forms on a $2s$-dimensional vector space.
Now we can use the classification of couples of symplectic forms given in \cite{GZ}. We need a part of
this classification given in Proposition \ref{prop-GZ} of the
present work.
The following theorem is a direct corollary of
this proposition and
Proposition \ref{prop-algebra-1}.

\begin{prop}
\label{prop-ei}
Each of the characteristic numbers of a regular tuple (\ref{tu-spaces-reduced})
 is different from $0$ and has multiplicity $\ge 2$. Consequently (\ref{tu-spaces-reduced}) has not more than $s$ distinct characteristic numbers; if $s=1$ then it has only one characteristic number.
The multiplicity of each of the characteristic numbers of a generic
 regular tuple (\ref{tu-spaces-reduced}) is equal to $2$ and consequently a generic regular tuple (\ref{tu-spaces-reduced}) has $s$ distinct characteristic numbers.
In this case (\ref{tu-spaces-reduced}) is isomorphic to another regular tuple of the same form if and only if
the two tuples have the same characteristic numbers.
 \end{prop}

The following statement is not more than a logical corollary of Proposition \ref{prop-ei} and Definition \ref{def-charact-numbers}, but it is worth to display it.

\begin{prop}
\label{prop-char-numbers} Let $2\le k\le 2n-2$.
The characteristic numbers of a regular tuple (\ref{tu})  have the same properties as in Proposition \ref{prop-ei} with $s = s(k,n) =
{\rm
min}\ \big([k/2], [(2n-k)/2]\big)$.
\end{prop}

\subsection{Characteristic Hamiltonians}
In the case $k>n$ and under the genericity assumption that each of the characteristic numbers has minimal possible multiplicity $2$ and
consequently we have $s = s(k,n) = [(2n-k)/2]$ distinct characteristic numbers $\lambda _1,...,\lambda _s$,
the characteristic numbers can be extended to functions on the symplectic manifold
\begin{equation}
\label{Q}
\left(Q, \omega _Q\right), \ \ \ \ Q = S_1^k\cap S_2^k, \ \omega _Q = \omega \vert _{TQ}
\end{equation}
(the fact that it is symplectic follows from the regularity of a tuple (\ref{tu})) by associating to
a point $z\in Q$, close to $0\in \mathbb R^{2n}$, the characteristic numbers of the linearization of
(\ref{tu}) at $z$. We obtain $s$ smooth functions $h_1, \dots , h_s$ on the symplectic manifold (\ref{Q}) taking
the values $\lambda _1,...,\lambda _s$ at $z=0$.

\begin{defn}
Let $k>n$.
The constructed functions $h_1,...,h_s$, $s = s(k,n) = [(2n-k)/2]$ on the symplectic manifold (\ref{Q}) will be called characteristic Hamiltonians of a regular tuple (\ref{tu}).
\end{defn}

It is worth to note that this definition works {\it only} under the assumption that
each of the characteristic numbers of a regular tuple (\ref{tu}) has minimal possible multiplicity
$2$ so that the linearization of (\ref{tu}) at any point $z\in Q$ close to $0$ has {\it the same} number $s = s(k,n) = [(2n-k)/2]$ of distinct characteristic numbers.

\section{Theorems on complete system of invariants}
\label{sec-main-results}

\subsection{The case ${\bf 2\le k\le n}$}

\begin{thm}
\label{thm-A}
Let $2\le k\le n$. Assume that the characteristic numbers of two regular tuples (\ref{tu})
have minimal possible multiplicity $2$ and
consequently each of the tuples has $[k/2]$ distinct characteristic numbers. The tuples
are equivalent if and only if their characteristic numbers are the same.
\end{thm}

\begin{proof}
The ``only if" part holds without the assumption on the
multiplicities and it is a part of Theorem \ref{thm-linearization}. The ``if"  part
is a direct corollary of the same Theorem \ref{thm-linearization}, Proposition \ref{prop-ei}, and
Proposition \ref{prop-reduced-tuple}. In fact, if
the characteristic numbers of two tuples $T$ and $\widetilde T$ of form (\ref{tu})
are the same and have
minimal multiplicity $2$ then by Proposition \ref{prop-ei} the reduced linearizations at $0$ of $T$ and $\widetilde T$
are isomorphic, by Proposition \ref{prop-reduced-tuple} their linearization at $0$ are also isomorphic,
and by Theorem \ref{thm-linearization} the tuples are isomorphic.
\end{proof}

 Note that the assumption on the multiplicities in Theorem \ref{thm-A} always holds for $k=2$ and $k=3$ when we have only one characteristic number. In the case $k=2$ Theorem \ref{thm-A} was proved in our work \cite{DJZ2} section 7.4,
where the characteristic number  is called there the  index of non-orthogonality between
 $S^2_1$ and $S^2_2$.

\subsection{The case ${\bf n<k\le 2n-2}$}
Consider two regular tuples of form (\ref{tu}):
\begin{equation}
\label{tutu}
T = \left(\mathbb R^{2n}, \ \omega , \ S_1^k\cup S_2^k\right)_0, \ \ \widetilde T = \left(\mathbb R^{2n}, \ \widetilde \omega , \ \widetilde S_1^k\cup \widetilde S_2^k\right)
\end{equation}
and the symplectic manifolds
\begin{equation}
\label{QQ}
\begin{split}
(Q, \omega _Q), \ \ Q = S_1^k\cap S_2^k, \ \ \omega _Q = \omega \vert _{TQ}\\
(\widetilde Q, \widetilde \omega  _{\widetilde Q}), \ \
\widetilde Q = \widetilde S_1^k\cap \widetilde S_2^k,
\ \ \widetilde \omega _{\widetilde Q} = \widetilde \omega \vert _{T\widetilde Q}.
\end{split}
 \end{equation}

\begin{thm}
\label{thm-B} Let $n<k\le 2n-2$. Assume that the characteristic numbers of two regular tuples  (\ref{tutu})   have minimal possible multiplicity $2$ and
consequently each of the tuples has $s=s(k,n) = [(2n-k)/2]$ distinct characteristic numbers and the characteristic Hamiltonians $h_1,...,h_s$ and $\widetilde h_1,...,\widetilde h_s$ are well-defined.
The tuples $T$ and $\widetilde T$ are equivalent if and only if there exists a local diffeomorphism
$\phi: Q\to \widetilde Q$ which sends $\widetilde \omega _{\widetilde Q}$ to $\omega _Q$ and the tuple of functions $(\widetilde h_1,...,\widetilde h_s)$ to $(h_1,...,h_s)$.
\end{thm}

Like in Theorem \ref{thm-A}, the assumption on multiplicities always holds if $k = 2n-3$ or $k=2n-2$ when we have only one characteristic number.

\subsubsection{Proof of the ``only if" part}
Assume that the tuples (\ref{tutu}) are equivalent via a local diffeomorphism $\Phi $ of $\mathbb R^n$. Since $\Phi $ sends $S_1^k$ to $\widetilde S_1^k$ and $S_2^k$ to $\widetilde S_2^k$ it sends $Q$ to
$\widetilde Q$. It also sends $\widetilde\omega $ to $\omega $ and consequently the restriction $\phi $ of $\Phi $ to $Q$ sends the form $\widetilde \omega _{\widetilde Q}$ to the form $\omega _Q$. The differential of the diffeomorphism $\Phi $ at a point $z\in Q$ sends the linearization of $T$ at $z$ to the linearization of $\widetilde T$ at the point $\phi (z)$. Therefore these two linearizations are isomorphic. By Proposition \ref{prop-reduced-tuple} the corresponding reduced linearizations are also isomorphic. By Theorem \ref{thm-char-numbers} these reduced linearizations have the same characteristic numbers. Therefore $\widetilde h_i(\phi (z)) = h_i(z)$, up to numeration.

\subsubsection{Proof of the ``if" part}
The proof of the ``if" part is reduction to Theorem \ref{thm-basic-B}. We will assume, without loss of generality, that $S_1^k=\widetilde S_1^k$ and $S_2^k = \widetilde S_2^k$. Let $\phi $ be a local diffeomorphism of $Q$
which brings $\widetilde h_i$ to $h_i$, up to numeration. We can extend $\phi $ to a local diffeomorphism $\Psi $
of $\mathbb R^{2n}$ which preserves $S_1^k$ and $S_2^k$. Applying $\Psi $
to the tuple $\widetilde T$ we obtain a tuple with characteristic Hamiltonians coinciding with those of the tuple $T$, up to numeration. It reduces the proof to the case that $T$ and $\widetilde T$ satisfy the following conditions:

\medskip

\noindent (a) \  $S_i^k=\widetilde S_i^k$ and consequently $Q  =\widetilde Q$;

\smallskip

\noindent (b)\  the reduced linearizations of (\ref{tutu}) at any point $z\in Q$

\hskip .15cm have the same characteristic numbers;

\smallskip

\noindent (c) \ $\omega $ and $\widetilde \omega $ have the same restriction to the tangent bundle of $Q$.

   \medskip

    By Propositions \ref{prop-ei} and \ref{prop-reduced-tuple} there is a family of isomorphisms
$\tau _z: T_z\mathbb R^{2n}\to T_z\mathbb R^{2n}$, parameterized by a point $z\in Q$, which brings the linearization of $T$ at $z\in Q$ to the linearization of $\widetilde T$ at the same point $z$. Condition (c) allows to chose $\tau _z$ such that it preserves $T_zQ$ and its restriction to $T_zQ$ is the identity map, for any $z\in Q$. Having a family of isomorphism $\tau _z$ with this property, we can construct a local diffeomorphism $\Phi $ of $\mathbb R^{2n}$
which preserves $S_1^k$ and $S_2^k$ pointwise (and consequently preserves $Q$ pointwise) and such that $d\Phi \vert _z = \tau _z$ for any $z\in Q$. Applying this diffeomorphism $\Phi $ to the tuple $\widetilde T$ we obtain
a tuple $\widehat T$ such that $T$ and $\widehat T$ have the same linearization at any point $z\in Q$.
Now the equivalence of the tuples follows from Theorem \ref{thm-basic-B}.

\subsection{The cases ${\bf k=2n-3, k=2n-2}$}
A short formulation of Theorem \ref{thm-B} is to say that under the given condition on multiplicity of the
characteristic numbers the tuple of characteristic Hamiltonians defined up to a symplectomorphism of
the symplectic manifold (\ref{Q}) is a complete invariant of a regular tuple (\ref{tu}) with $n<k\le 2n-2$.
Nevertheless, strictly speaking, Theorem \ref{thm-B} is a reduction theorem rather than a theorem on a complete system of invariants. It reduces the classification of generic tuples (\ref{tu}) with $n<k\le 2n-2$
to the classification of $[(2n-k)/2]$ functions on a symplectic manifold of dimension $2(n-k)$ with
respect to local symplectomorphisms of this manifold. It is well known that a single non-singular
function $h$ (such that $dh(0)\ne 0)$ can be reduced to $h(0)+z_1$ where $z_1$ is one of local coordinates.
Therefore Theorem \ref{thm-B} implies the following corollary.

\begin{thm}
\label{thm-C} Let $k = 2n-2\ge 4$ or $k=2n-3\ge 5$ so that
tuples (\ref{tutu}) have only one characteristic number $\lambda $ and $\widetilde \lambda $.
Assume that the characteristic Hamiltonians $h$ and $\widetilde h$ are non-singular: $dh(0)\ne 0$ and $d\widetilde h(0)\ne 0$. The tuples (\ref{tutu}) are equivalent if and only if $\lambda = \widetilde \lambda $.
\end{thm}

\subsection{Normal forms}

Using Theorems \ref{thm-A} - \ref{thm-C} it is easy to construct the following normal forms.
If $2\le k\le n$ then in suitable local coordinates $x,y\in \mathbb R^k, \ p,q \in \mathbb R^{n-k}$
a tuple (\ref{tu}) satisfying the assumptions of Theorem \ref{thm-A} has the form
\begin{equation}
\label{nf-A}
\begin{split}
S_1^k = \{y=p=q=0\}, \ S_2^k = \{x=p=q=0\},  \hskip 1cm \\
\omega = \sum _{i=1}^kdx_idy_i + \sum _{i=1}^{n-k}dp_idq_i +
\sum _{i=1}^sdx_{2i-1}dx_{2i} + \sum _{i=1}^{[k/2]}\frac{dy_{2i-1}dy_{2i}}{\lambda _i}
\end{split}
\end{equation}
If $n<k\le 2n-2$ then in suitable local coordinates $x,y\in \mathbb R^{2n-k}, \ p,q\in \mathbb R^{k-n}$ a tuple (\ref{tu}) satisfying the assumptions of Theorem \ref{thm-B} has the form
\begin{equation}
\label{nf-B}
\begin{split}
S_1^k = \{y=0\}, \ S_2^k = \{x=0\},\hskip 3cm \\
\omega = \sum _{i=1}^{2n-k}dx_i  dy_i + \sum _{i=1}^{k-n}dp_i dq_i +
\sum _{i=1}^{[(2n-k)/2]}dx_{2i-1} dx_{2i} + \sum _{i=1}^{[(2n-k)/2]}\frac{dy_{2i-1} dy_{2i}}{h_i(p,q)}
\end{split}
\end{equation}
The parameters $\lambda _i$ in normal form (\ref{nf-A})
 are moduli, and they are
exactly the characteristic numbers. The functional parameters $h_i(p,q)$ in normal form (\ref{nf-B})
are exactly the characteristic Hamiltonians.
 In the case that some of the characteristic numbers $\lambda _i = h_i(0)$ are not real
these normal forms  hold in complex coordinates. Namely, if $\lambda _i = \bar \lambda _j\not\in \mathbb R$ \ then
\ $x_{2i-1} = \bar x_{2j-1}, \ x_{2i} = \bar x_{2j}, \ y_{2i-1} = \bar y_{2j-1}, \ y_{2i} = \bar y_{2j}$ are complex valued
conjugate coordinates and $h_i(p,q) = \bar h_j(p,q)$ are complex valued conjugate functions.

\medskip

If $k = 2n-2\ge 4$ or $k=2n-3\ge 5$ and a tuple (\ref{tu}) satisfies the assumption of Theorem \ref{thm-C} then in
suitable coordinates it has the form (\ref{nf-B}) with $h_1(p,q)\equiv \lambda _1$, i.e. with only one parameter
$\lambda _1$.

\subsection{The case $\bf{ n<k\le 2n-4}$. Functional moduli} Note that this case is possible only if the dimension of the symplectic space $(\mathbb R^{2n}, \omega )$ is at least $10$.
Theorem {\ref{thm-B} implies that a generic tuple (\ref{tu}) has in suitable local coordinates the normal form (\ref{nf-B}) with $h_1(p,q)\equiv \lambda _1$. (The genericity condition are the assumption of Theorem \ref{thm-B} and
the requirement hat at least one of the characteristic Hamiltonians is a non-singular function).
This normal form is parameterized by $s-1$ functions $h_2(u,v), ...,h_s(u,v)$, $s = [(2n-k)/2]\ge 2$.
Since the group of local symplectomorphisms can be parameterized by one function, it is almost clear that this normal form is asymptotically exact in the following sense.

\begin{defn}
\label{def-funct-moduli}
Let $m_\ell $ be the number of moduli in the classification of generic germs
(in any classification problem of local analysis). Assume that $m_{\ell }\to \infty $ as $\ell \to \infty $.
A normal form, parameterized by functions,  is called asymptotically exact if the
number of parameters $p_\ell $ of its $\ell $-jet satisfies $p_\ell = m_\ell (1+o(1))$ as $\ell \to \infty $.
\end{defn}

With this definition, we obtain one more, the following  corollary of Theorem \ref{thm-B}.

\begin{thm}
\label{thm-D}
Let $n<k\le 2n-4$
so that $s = s(k,n) = [(2n-k)/2]\ge 2$. In this case the number of moduli in the classification of $\ell $-jets of generic tuples (\ref{tu}) goes to $\infty $ as $\ell \to \infty $.
A generic tuple (\ref{tu}) has in suitable coordinates normal form (\ref{nf-B}) with $h_1(u,v)\equiv \lambda _1$,
parameterized by $(s-1)$ functions of \ $2(k-n)$ variables. This normal form is asymptotically exact.
\end{thm}

In the beginning of the paper, in Theorem \ref{thm-moduli} we stated that in the case of dimensions $k,n$ in Theorem \ref{thm-D} the functional moduli
are $s-1$ functions of $2(k-n)$ variables. Theorem \ref{thm-D}  gives a precise meaning of what we mean by these words. A more detailed characterization of ``functional codimension" of orbits in classification problems of local analysis requires Poincare series of moduli numbers which was introduced by V. Arnol'd in \cite{Ar2}.

\bibliographystyle{amsalpha}

\begin{thebibliography}{AAA}



\bibitem [Ar1] {Ar1} V.I. Arnol'd, \emph{First step of local symplectic algebra}, Differential topo\-logy, infinite-dimensional Lie algebras, and
applications.
AMS Transl., Ser. 2,
 194(44), 1999, 1-8.

\smallskip

\bibitem [Ar2] {Ar2} V.I. Arnol'd,
\emph{Mathematical problems in classical physics}.
Trends and perspectives in applied mathematics, 1–20,
Appl. Math. Sci., 100, Springer, New York, 1994.

\smallskip

\bibitem [AG] {AG} V.I. Arnol'd, A. B. Givental'
{\em Symplectic geometry},
Dynamical systems, IV,
Encycl.of Math. Sciences, vol. 4,
Springer, Berlin, 2001.

\smallskip

\bibitem [DJZ1] {DJZ1} W. Domitrz, S. Janeczko, M. Zhitomirskii,
\emph{Relative Poincare lemma, contractibility, quasi-homogeneity}
\emph{and vector fields tangent to a singular variety},
Illinois J. Math.
48, No.3 (2004), 803-835.

\smallskip

\bibitem [DJZ2] {DJZ2} W. Domitrz, S. Janeczko, M. Zhitomirskii,
\emph{Symplectic singularities of varietes: the method of
algebraic restrictions},
J.  f\"{u}r die reine und angewandte
Mathematik 618(2008), 197--235.

\smallskip

\bibitem [GZ] {GZ} I. M. Gelfand, I. S. Zakharevich,
\emph{Spectral
theory of a pencil of third-order skew-symmetric differential
operators on $s^1$},
Funct. Anal. Appl. 23 (1989), no. 2, 85–93

\smallskip

\bibitem [Ma] {Ma} J. Martinet,
\emph{Sur les singularites des formes differentielles},
Ann. Inst. Fourier 20 (1970), no. 1, 95-178.

\smallskip

\bibitem [MR] {MR} A.S. McRae,
\emph{Symplectic geometry for
pairs of submanifolds},
Rocky Mountain J. Math. 35(2005), no. 5,
1755-1764.

\smallskip

\bibitem[Me] {Me} R. B. Melrose,
\emph{Equivalence of glancing hypersurfaces},
Invent. Math. 37 (1976), 165-191.

\smallskip

\bibitem[Os] {Os} T. Oshima,
\emph{On analytic equivalence of glancing hypersurfaces},
Sci. Pap. Coll. Gen. Educ., Univ. Tokyo 28 (1978), 51-57 .

\smallskip

\bibitem [Zh1] {Zh1} M. Zhitomirskii,
\emph{ Relative Darboux
theorem for singular manifolds and local contact algebra},
Canadian J.
Math. \textbf{57}, No.6 (2005), 1314-1340.

\smallskip

\bibitem [Zh2] {Zh2} M. Zhitomirskii,
\emph{Typical singularities of differential $1$-forms and Pfaffian equations},
Translations of Mathematical Monographs, 113. AMS, Providence, RI, 1992.

\end{thebibliography}

\end{document}